\documentclass[11pt]{amsart}

\usepackage[T1]{fontenc}
\usepackage{lmodern}
\usepackage{microtype}
\usepackage{mathtools}
\usepackage{amssymb}
\usepackage[hidelinks]{hyperref}

\newtheorem{theorem}{Theorem}[section]
\newtheorem{lemma}[theorem]{Lemma}
\newtheorem{proposition}[theorem]{Proposition}
\newtheorem{corollary}[theorem]{Corollary}
\theoremstyle{remark}
\newtheorem{remark}[theorem]{Remark}
\newtheorem{conjecture}[theorem]{Conjecture}

\numberwithin{equation}{section}

\newcommand{\Nzero}{\mathbb N_0}

\newcommand{\dd}{\,d}

\title[The EMN and Krasikov conjectures]
{The Erd\'elyi--Magnus--Nevai and Krasikov Conjectures for Jacobi Polynomials}
\author{Qi-Feng Bai}
\address{School of Mathematics and Statistics, Nanfang College, Guangzhou, 510970, Guangdong, China}
\email{baiqifeng@nfu.edu.cn}
\author{Yu-Tian Li}
\address{School of Mathematics and Statistics, Nanfang College, Guangzhou, 510970, Guangdong, China}
\email{yutianlee@gmail.com}
\date{}

\subjclass[2020]{Primary 33C45; Secondary 26D05, 41A17, 42C05}
\keywords{Jacobi polynomials, Bernstein-type inequality, Erd\'elyi--Magnus--Nevai
conjecture, Krasikov conjecture, Christoffel function, contiguous relations}

\hypersetup{
  pdftitle={The Erdelyi--Magnus--Nevai and Krasikov Conjectures for Jacobi Polynomials},
  pdfauthor={Qi-Feng Bai and Yu-Tian Li}}

\begin{document}

\begin{abstract}
Let $p_n^{(\alpha,\beta)}$ denote the Jacobi polynomial orthonormal for the weight
$(1-x)^\alpha(1+x)^\beta$ on $[-1,1]$, where
$\alpha,\beta\ge-1/2$, and put $S=\alpha+\beta+1$. We prove the uniform
degree--parameter estimate
\[
 (1-x)^{\alpha+1/2}(1+x)^{\beta+1/2}
 \bigl|p_n^{(\alpha,\beta)}(x)\bigr|^2
 \le C\max\bigl\{1,S^{1/3},S^{1/2}(n+1)^{-1/6}\bigr\}.
\]
This proves, in an equivalent symmetric parametrisation, the stronger
degree-sensitive conjecture proposed by Krasikov and implies the
Erd\'elyi--Magnus--Nevai conjecture. The proof starts from Krasikov's estimate in
the high-parameter quadrant and transports it to the hard edges through weighted
contiguous relations whose singular endpoint terms cancel; direct hypergeometric
estimates control the remaining endpoint caps. A Bessel turning-point argument
shows that the intermediate factor $S^{1/3}$ in the squared estimate cannot be
omitted. We also derive degree-sensitive lower bounds for Jacobi Christoffel
functions and Gauss--Jacobi quadrature weights.
\end{abstract}

\subjclass[2020]{Primary 33C45; Secondary 41A17, 42C05}

\keywords{Jacobi polynomials, Bernstein-type inequalities,
Erd\'elyi--Magnus--Nevai conjecture, Krasikov conjecture,
uniform degree--parameter estimates, Christoffel functions,
contiguous relations, Gauss--Jacobi quadrature}

\maketitle

\section{Introduction}\label{sec:intro}

Let
\[
w_{\alpha,\beta}(x)=(1-x)^\alpha(1+x)^\beta,\qquad -1<x<1,
\]
let $P_n^{(\alpha,\beta)}$ be the Jacobi polynomial in its standard
normalisation, and let $p_n^{(\alpha,\beta)}$ be the polynomial of degree $n$
orthonormal in $L^2([-1,1],w_{\alpha,\beta}\dd x)$. A Bernstein-type inequality
for these polynomials estimates
\begin{equation}\label{eq:Fdef}
F_{n,\alpha,\beta}(x)
=(1-x^2)^{1/4}\,w_{\alpha,\beta}(x)^{1/2}\,
\bigl|p_n^{(\alpha,\beta)}(x)\bigr|,
\end{equation}
the polynomial damped both by the square root of its own weight and by the factor
that compensates the oscillatory singularity at the endpoints. Bernstein's
classical inequality for Legendre polynomials is the case $\alpha=\beta=0$; for
general parameters the subject was developed by Chow, Gatteschi and Wong
\cite{ChowGatteschiWong}, by Gautschi \cite{Gautschi}, and, in the direction that
concerns us, by Erd\'elyi, Magnus and Nevai \cite{EMN}.

Writing $F=F_{n,\alpha,\beta}$ for brevity, one has the exact identity
\begin{equation}\label{eq:Fsquare}
F(x)^2=(1-x)^{\alpha+1/2}(1+x)^{\beta+1/2}\bigl|p_n^{(\alpha,\beta)}(x)\bigr|^2 ,
\end{equation}
and it is in this form that the problem is usually stated. Erd\'elyi, Magnus and
Nevai proved \cite[Thm.~1]{EMN} that for all $n\ge0$ and all
$\alpha,\beta\ge-1/2$,
\begin{equation}\label{eq:EMNthm}
\max_{-1\le x\le1}F_{n,\alpha,\beta}(x)^2
\ =\ O\bigl(\max\{1,(\alpha^2+\beta^2)^{1/2}\}\bigr),
\end{equation}
with an explicit absolute implied constant, and they conjectured that the
exponent $1/2$ in \eqref{eq:EMNthm} may be halved:
\begin{equation}\label{eq:EMNconj}
\max_{-1\le x\le1}F_{n,\alpha,\beta}(x)^2
=O\bigl(\max\{1,(\alpha^2+\beta^2)^{1/4}\}\bigr),
\end{equation}
uniformly in $n\ge0$ and $\alpha,\beta\ge-1/2$. This is the
Erd\'elyi--Magnus--Nevai (EMN) conjecture.

Krasikov subsequently proposed a stronger, degree-sensitive estimate
\cite{KrasikovMaximum}, restated as Conjecture~2 in \cite{KrasikovEMN}. After
reflection, write $a=\max\{\alpha,\beta\}$; for $n\ge1$ his conjecture is
\begin{equation}\label{eq:krasikov-conjecture}
 \max_{-1\le x\le1}F_{n,\alpha,\beta}(x)^2
 =O\!\left(\max\left\{1,|a|^{1/3}
 \left(1+\frac{|a|}{n}\right)^{1/6}\right\}\right).
\end{equation}
To state our result symmetrically and include degree zero, put
\begin{equation}\label{eq:envelope-def}
 \begin{gathered}
 S=\alpha+\beta+1,\\
 \Phi_n(S)=\max\left\{1,S^{1/3},\frac{S^{1/2}}{(n+1)^{1/6}}\right\},\\
 \mathcal E_n(S)=\Phi_n(S)^{1/2}.
 \end{gathered}
\end{equation}
For $n\ge1$, $\Phi_n(S)$ is equivalent, up to absolute constants, to the
right-hand side of \eqref{eq:krasikov-conjecture}; see
Remark~\ref{rem:krasikov-equivalence}. Thus the main theorem below proves
Krasikov's conjecture in its big-$O$ sense and gives its natural degree-zero
extension. Since $\mathcal E_n(S)\le\max\{1,S^{1/4}\}$, it also proves the EMN
conjecture. The separate terms in \eqref{eq:envelope-def} record three regimes:
bounded size, the Bessel turning-point scale $S^{1/6}$, and the fixed-degree
scale $S^{1/4}(n+1)^{-1/12}$.

\subsection*{What was known}
Set
\begin{equation}\label{eq:A0}
A_0=\frac{1+\sqrt2}{4}=0.60355\ldots
\end{equation}
Krasikov \cite[Thm.~2]{KrasikovUpper} proved that
$\max_xF_{n,\alpha,\beta}^2<3\alpha^{1/3}(1+\alpha/n)^{1/6}$ whenever $n\ge6$ and
$\alpha\ge\beta\ge A_0$; this has precisely the degree--parameter order in
\eqref{eq:envelope-def} on the quadrant $[A_0,\infty)^2$. His later paper
\cite{KrasikovEMN} proved the conjectured order in the ultraspherical case
$\alpha=\beta$. Haagerup and Schlichtkrull \cite{HaagerupSchlichtkrull} obtained
an inequality valid for all $n\ge0$ and all $\alpha,\beta\ge0$, but in
orthonormal form it carries a factor $(2n+\alpha+\beta+1)^{1/4}$ and so does not
give a degree-uniform statement. Christiansen and Rubin
\cite[\S3]{ChristiansenRubin}, in their work on Chebyshev polynomials for Jacobi
weights, observed that their Theorem~3.1 yields a uniform bound when
$|\alpha|,|\beta|\le1/2$; because $S$ is bounded there, this is also the order in
\eqref{eq:envelope-def}.

Thus the required estimate was known on
\begin{equation}\label{eq:known}
[A_0,\infty)^2\ \cup\ [-\tfrac12,\tfrac12]^2 ,
\end{equation}
but, to the best of our knowledge, not on the complement of \eqref{eq:known} in
$\{\alpha,\beta\ge-\frac12\}$: two unbounded \emph{mixed strips}, in which one
parameter is large while the other lies near the hard edge $-1/2$, together with
two small rectangles adjacent to the square. The mixed strips are where the
problem is genuinely hard. The standard uniform asymptotic estimates used for
Jacobi polynomials lose uniformity as $\beta\downarrow-1/2$, where the endpoint
weight exponent $\beta+1/2$ tends to zero; and the obvious remedy---raise
$\beta$ into Krasikov's range using a
contiguous relation---fails because the relations available do not respect the
weight factor $(1+x)^{\beta/2}$ in \eqref{eq:Fdef}, and produce coefficients that
blow up exactly at $x=-1$.

Two lines of the strips do reduce to what is known, by a mechanism worth
recording: the quadratic transformations of the Jacobi polynomials preserve
$F$ exactly, so that the lines $\beta=-1/2$ and $\beta=1/2$ carry the even,
respectively odd, ultraspherical case. This is Proposition~\ref{prop:quadratic}
below. The remaining, two-dimensional part of each strip is not covered by
these reductions.

\subsection*{Results}
Values of $F$ at $x=\pm1$ are always understood as one-sided limits from
$(-1,1)$; by \eqref{eq:Fsquare} these exist and are finite.

\begin{theorem}[Degree--parameter envelope]\label{thm:main}
For all $n\in\Nzero$, all $\alpha,\beta\ge-1/2$ and all $x\in[-1,1]$,
\begin{equation}\label{eq:main}
 F_{n,\alpha,\beta}(x)\le C_*\,\mathcal E_n(\alpha+\beta+1),
 \qquad C_*=1.2\cdot10^5.
\end{equation}
Equivalently,
\[
 F_{n,\alpha,\beta}(x)^2\le C_*^2
 \max\left\{1,S^{1/3},\frac{S^{1/2}}{(n+1)^{1/6}}\right\},
 \qquad S=\alpha+\beta+1.
\]
No attempt is made to optimise the numerical constant.
\end{theorem}

\begin{corollary}[EMN conjecture]\label{cor:EMN}
Put $M(\alpha,\beta)=\max\{1,(\alpha+\beta+1)^{1/4}\}$. Then
\begin{equation}\label{eq:EMN-corollary}
 F_{n,\alpha,\beta}(x)\le C_*M(\alpha,\beta)
\end{equation}
uniformly in $n$, $\alpha$, $\beta$ and $x$. Consequently
\eqref{eq:EMNconj} holds.
\end{corollary}

\begin{proof}
Each of the three terms defining $\mathcal E_n(S)$ is at most
$\max\{1,S^{1/4}\}$, so \eqref{eq:EMN-corollary} follows from
Theorem~\ref{thm:main}. Squaring and using Remark~\ref{rem:equivalence} gives
\eqref{eq:EMNconj}.
\end{proof}

The next result records two rigorously sharp regimes. The hard-edge formula at
fixed degree forces the factor $S^{1/4}(n+1)^{-1/12}$, while a Bessel
turning-point sequence forces $S^{1/6}$. We do not claim here that the entire
three-regime envelope is pointwise optimal.

\begin{theorem}[Lower bounds]\label{thm:sharpness}
Write $x=\cos\theta$, $\theta\in[0,\pi]$.
\begin{enumerate}
\item[(i)] For every $n\ge0$,
\[
\begin{aligned}
F_{n,\frac12,-\frac12}(x)
 &=\sqrt{\tfrac2\pi}\,\bigl|\sin (n+\tfrac12)\theta\bigr|,\\
F_{n,-\frac12,\frac12}(x)
 &=\sqrt{\tfrac2\pi}\,\bigl|\cos (n+\tfrac12)\theta\bigr|,\\
F_{n,\frac12,\frac12}(x)
 &=\sqrt{\tfrac2\pi}\,\bigl|\sin (n+1)\theta\bigr|,
\end{aligned}
\]
while $F_{n,-\frac12,-\frac12}(x)=\sqrt{2/\pi}\,|\cos n\theta|$ for $n\ge1$ and
$F_{0,-\frac12,-\frac12}\equiv\pi^{-1/2}$. Apart from this degree-zero
exception, the maxima are $\sqrt{2/\pi}$. In particular every admissible
constant in either \eqref{eq:main} or \eqref{eq:EMN-corollary} is at least
$\sqrt{2/\pi}$.
\item[(ii)] For every fixed $n\in\Nzero$ and every $\alpha\ge-1/2$,
\begin{equation}\label{eq:sharp-exact}
F_{n,\alpha,-\frac12}(-1)^2
=\frac{(2n+\alpha+\tfrac12)\,\Gamma(n+\tfrac12)\,\Gamma(n+\alpha+\tfrac12)}
{\pi\,\Gamma(n+1)\,\Gamma(n+\alpha+1)},
\end{equation}
the right-hand side being read by continuous extension at
$(n,\alpha)=(0,-\frac12)$, where it equals $1/\pi$. Consequently
\begin{equation}\label{eq:sharp-asymptotic}
F_{n,\alpha,-\frac12}(-1)\sim
\Bigl(\frac{\Gamma(n+\tfrac12)}{\pi\,\Gamma(n+1)}\Bigr)^{1/2}\alpha^{1/4}
\qquad(\alpha\to\infty),
\end{equation}
so the fixed-degree exponent $1/4$ in $\mathcal E_n$ cannot be lowered.
\item[(iii)] For every fixed $A\ge-1/2$ and $z>0$,
\begin{equation}\label{eq:bessel-limit}
 \lim_{n\to\infty}F_{n,A,A}\!\left(\cos\frac zn\right)
 =\sqrt z\,|J_A(z)|.
\end{equation}
Moreover, there is an absolute constant $c>0$ such that, for every sufficiently
large $A$, there exists an integer $n_A$ with $n_A+1\ge(2A+1)^3$ for which
\begin{equation}\label{eq:bessel-lower}
 F_{n_A,A,A}\!\left(\cos\frac A{n_A}\right)\ge cA^{1/6}
\end{equation}
holds. Hence, along this sequence, $S=2A+1$, and the term $S^{1/6}$ in
$\mathcal E_n(S)$ cannot be omitted even though
$S^{1/4}(n+1)^{-1/12}\le1$.
\end{enumerate}
\end{theorem}

An application of Theorem~\ref{thm:main} lies close to the setting in which the
EMN conjecture was posed. For $n\in\Nzero$ let
\[
\begin{aligned}
\lambda_n^{(\alpha,\beta)}(x)
&=\min\Bigl\{\int_{-1}^1|Q|^2w_{\alpha,\beta}\dd t:
  \deg Q\le n,\ Q(x)=1\Bigr\}\\
&=\Bigl(\sum_{k=0}^n p_k^{(\alpha,\beta)}(x)^2\Bigr)^{-1}
\end{aligned}
\]
denote the Christoffel function of the weight $w_{\alpha,\beta}$.

\begin{theorem}[Uniform Christoffel bound]\label{thm:christoffel}
For all $n\in\Nzero$, all
$\alpha,\beta\ge-1/2$ and all $x\in(-1,1)$,
\begin{equation}\label{eq:christoffel}
\lambda_n^{(\alpha,\beta)}(x)
\ \ge\ \frac{w_{\alpha,\beta}(x)\sqrt{1-x^2}}
{3C_*^2\,(n+1)\,\Phi_n(\alpha+\beta+1)} .
\end{equation}
In particular the Gauss--Jacobi quadrature weights for
$w_{\alpha,\beta}$ obey the same lower bound at their nodes.
\end{theorem}

For bounded parameters this has the classical bulk order
$w_{\alpha,\beta}(x)\sqrt{1-x^2}/n$ (compare \cite{NevaiCase}). In contrast to
the EMN corollary, the factor in \eqref{eq:christoffel} retains the improvement
when the degree is large relative to the parameters.

\subsection*{Method}
The estimate we import is Krasikov's, on the quadrant $\alpha,\beta\ge A_0$; what
has to be constructed is a transport mechanism carrying that estimate down to the
hard edge without losing the weight. Two contiguous relations, taken separately,
each fail: one raises $\beta$ but changes the polynomial's endpoint behaviour, the
other restores the endpoint behaviour but reintroduces a factor $(1+x)^{-1/2}$.
Combined in the right proportion (Lemma~\ref{lem:cancellation}) the leading
endpoint singularities cancel, and what survives is an inequality
\[
F_{n,\alpha,\beta}\ \le\ A\,F_{n-1,\alpha,\beta+2}
+B\,(1+x)^{-1/2}F_{n-1,\alpha,\beta+1}
\]
in which $B$ carries a compensating factor $n^{-1}$. The residual singularity is
then harmless provided one never uses the inequality too close to $x=-1$; and
close to $x=-1$ one does not need it, because there the hypergeometric series at
the endpoint converges so fast that a direct estimate applies
(Lemma~\ref{lem:mixed-cap}). Matching the two regimes at the scale
$1+x\asymp[(n+1)(n+\alpha+1)]^{-1}$ is what makes the argument uniform in the
degree and preserves the full scale $\mathcal E_n(S)$. Two applications of the
mechanism carry any $\beta\in[-\frac12,A_0]$ into
$[A_0,A_0+2]$, and a two-sided version handles the low-parameter rectangle. Only
finitely many shifts occur, so every constant remains absolute; \S\ref{sec:proof}
collects them.

\subsection*{Organisation}
\S\ref{sec:normalisation} fixes normalisations and the envelope calculus.
\S\ref{sec:highhigh} records the high-parameter input.
\S\ref{sec:elementary} contains two elementary estimates
and \S\ref{sec:caps} the endpoint estimates. \S\ref{sec:transitions} constructs
the transport mechanism, \S\ref{sec:mixed} applies it in the mixed strips, and
\S\ref{sec:proof} completes the proof of Theorem~\ref{thm:main}.
\S\ref{sec:sharpness} proves Theorem~\ref{thm:sharpness} and the quadratic
transformation identity; \S\ref{sec:christoffel} proves
Theorem~\ref{thm:christoffel}; \S\ref{sec:open} states the remaining
sharp-constant problem.

\subsection*{Acknowledgement of a related manuscript}
Estimates of the present kind enter the analysis of the discrete Laguerre
operator carried out by Koornwinder, Kostenko and Teschl \cite{KKT}, who noted in
\cite[Rem.~6.4(ii)]{KKT} that \eqref{eq:EMNconj} would make one of their kernel
estimates unconditional. That circle of ideas is developed in the companion
manuscript \cite{LiKKT}. Nothing in the present paper depends on \cite{LiKKT};
the normalisation interface between the two is recorded in
Remark~\ref{rem:companion}.

\section{Normalisation and envelope calculus}\label{sec:normalisation}

For $\alpha,\beta>-1$ put
\begin{equation}\label{eq:norm}
H_n^{\alpha,\beta}
=\int_{-1}^1\bigl(P_n^{(\alpha,\beta)}\bigr)^2w_{\alpha,\beta}\dd x
=\frac{2^{\alpha+\beta+1}}{2n+\alpha+\beta+1}\cdot
\frac{\Gamma(n+\alpha+1)\Gamma(n+\beta+1)}{\Gamma(n+1)\Gamma(n+\alpha+\beta+1)},
\end{equation}
so that $p_n^{(\alpha,\beta)}=P_n^{(\alpha,\beta)}/\sqrt{H_n^{\alpha,\beta}}$. For
$n=0$ the right-hand side of \eqref{eq:norm} is indeterminate exactly when
$\alpha+\beta+1=0$, and we then use the beta integral
\begin{equation}\label{eq:norm0}
H_0^{\alpha,\beta}=2^{\alpha+\beta+1}B(\alpha+1,\beta+1),
\end{equation}
which is valid for all $\alpha,\beta>-1$ and agrees with \eqref{eq:norm} whenever
the latter is defined. All degree-zero computations below use
\eqref{eq:norm0}.

The function $F_{n,\alpha,\beta}$ was defined in \eqref{eq:Fdef}; the equivalent
form \eqref{eq:Fsquare} is the one to use at the endpoints, where \eqref{eq:Fdef}
may present a formal product $0\cdot\infty$ while \eqref{eq:Fsquare} has a finite
one-sided limit. From $P_n^{(\alpha,\beta)}(-x)=(-1)^nP_n^{(\beta,\alpha)}(x)$ and
$H_n^{\alpha,\beta}=H_n^{\beta,\alpha}$ we get the reflection identity
\begin{equation}\label{eq:reflection}
F_{n,\alpha,\beta}(x)=F_{n,\beta,\alpha}(-x),
\end{equation}
used whenever a statement about one endpoint or one mixed strip is transferred to
its mirror image.

\begin{remark}[Equivalent form of Krasikov's scale]\label{rem:krasikov-equivalence}
For $S\ge0$ and $n\in\Nzero$,
\begin{equation}\label{eq:Phi-equivalence}
 \begin{aligned}
 \Phi_n(S)&\le
 \max\left\{1,S^{1/3}\left(1+\frac{S}{n+1}\right)^{1/6}\right\}\\
 &\le2^{1/6}\Phi_n(S).
 \end{aligned}
\end{equation}
For $S=0$ the assertion is immediate. For $S>0$, after factoring out
$S^{1/3}$ it is the elementary inequality
\[
 \max\{1,t^{1/6}\}\le(1+t)^{1/6}
 \le2^{1/6}\max\{1,t^{1/6}\}.
\]
For $n\ge1$, reflect so that
$a=\alpha\ge\beta$. If $a\ge1$, then
$a\le S\le3a$ and $n\le n+1\le2n$; hence the middle expression in
\eqref{eq:Phi-equivalence} is comparable, with absolute constants, to the
right-hand side of \eqref{eq:krasikov-conjecture}. If $a<1$, both expressions,
including their outer maximum with $1$, remain between positive absolute
constants. This proves the asserted equivalence with Krasikov's formulation.
\end{remark}

\begin{remark}\label{rem:equivalence}
For $\alpha,\beta\ge-1/2$,
\[
\begin{aligned}
2^{-1/4}\max\{1,(\alpha^2+\beta^2)^{1/4}\}
&\le\max\{1,(\alpha+\beta+1)^{1/2}\}\\
&\le\sqrt3\,\max\{1,(\alpha^2+\beta^2)^{1/4}\}.
\end{aligned}
\]
Indeed, let $m=\max(\alpha,\beta)$. If $m\ge1$ then
$m\le\alpha+\beta+1\le2m+1\le3m$ while $m\le\sqrt{\alpha^2+\beta^2}\le\sqrt2\,m$,
which gives both inequalities. If $m<1$ then $\max\{1,(\alpha+\beta+1)^{1/2}\}$ and
$\max\{1,(\alpha^2+\beta^2)^{1/4}\}$ both lie in $[1,\sqrt3)$, so the left-hand
side is at most $2^{-1/4}\cdot2^{1/4}=1$ and the right-hand side at least
$\sqrt3$. This is the comparison used in Corollary~\ref{cor:EMN} to recover the
usual formulation \eqref{eq:EMNconj}.
\end{remark}

\begin{lemma}[Stability under the parameter shifts]\label{lem:shift}
Let $s_0=A_0+\tfrac12$. If $n\ge1$, $S\ge s_0$ and $j\in\{1,2\}$, then
\begin{equation}\label{eq:shift}
 \mathcal E_{n-1}(S+j)\le1.4\,\mathcal E_n(S).
\end{equation}
\end{lemma}

\begin{proof}
The constant term causes no loss. For the second term,
\[
 (S+j)^{1/6}\le(1+2/s_0)^{1/6}S^{1/6}<1.2S^{1/6}.
\]
For the third term, comparison with the corresponding term of
$\mathcal E_n(S)$ gives the ratio
\[
 \left(1+\frac jS\right)^{1/4}\left(\frac{n+1}{n}\right)^{1/12}
 \le(1+2/s_0)^{1/4}2^{1/12}<1.372<1.4.
\]
Taking the maximum proves \eqref{eq:shift}.
\end{proof}

\section{The high-parameter quadrant}\label{sec:highhigh}

\begin{proposition}[Krasikov]\label{prop:krasikov}
For $n\ge6$ and $\alpha\ge\beta\ge A_0$,
\begin{equation}\label{eq:krasikov}
\max_{-1\le x\le1}
(1-x)^{\alpha+1/2}(1+x)^{\beta+1/2}\bigl|p_n^{(\alpha,\beta)}(x)\bigr|^2
<3\alpha^{1/3}\Bigl(1+\frac\alpha n\Bigr)^{1/6}.
\end{equation}
\end{proposition}

This is \cite[Thm.~2]{KrasikovUpper}, stated there for the orthonormal
polynomials, whose squared norm is \eqref{eq:norm}; the maximum is over the
closed interval, and no limiting convention is needed since $\alpha,\beta>0$.

\begin{proposition}[Haagerup--Schlichtkrull]\label{prop:HS}
For $n\in\Nzero$, $\alpha,\beta\ge0$ and $x\in[-1,1]$ put
\begin{equation}\label{eq:gdef}
g_n^{(\alpha,\beta)}(x)
=\Bigl(\frac{\Gamma(n+1)\Gamma(n+\alpha+\beta+1)}
{\Gamma(n+\alpha+1)\Gamma(n+\beta+1)}\Bigr)^{1/2}
\Bigl(\frac{1-x}2\Bigr)^{\alpha/2}\Bigl(\frac{1+x}2\Bigr)^{\beta/2}
P_n^{(\alpha,\beta)}(x).
\end{equation}
Then
\[
(1-x^2)^{1/4}\bigl|g_n^{(\alpha,\beta)}(x)\bigr|
\le C_{\mathrm{HS}}\,(2n+\alpha+\beta+1)^{-1/4},
\qquad C_{\mathrm{HS}}<12 .
\]
Moreover
\begin{equation}\label{eq:bridge}
F_{n,\alpha,\beta}(x)
=\Bigl(\frac{2n+\alpha+\beta+1}2\Bigr)^{1/2}(1-x^2)^{1/4}
\bigl|g_n^{(\alpha,\beta)}(x)\bigr|
\qquad(n\in\Nzero),
\end{equation}
so that
\begin{equation}\label{eq:HSorthonormal}
F_{n,\alpha,\beta}(x)\le 2^{-1/2}C_{\mathrm{HS}}\,(2n+\alpha+\beta+1)^{1/4}.
\end{equation}
\end{proposition}

\begin{proof}
The inequality, with the value $C_{\mathrm{HS}}<12$, is
\cite[Thm.~1.1]{HaagerupSchlichtkrull}. For \eqref{eq:bridge},
solve \eqref{eq:gdef} for
$(1-x)^{\alpha/2}(1+x)^{\beta/2}P_n^{(\alpha,\beta)}$ and divide by
$\sqrt{H_n^{\alpha,\beta}}$: the gamma quotients cancel and
$2^{(\alpha+\beta)/2}/\sqrt{2^{\alpha+\beta+1}/(2n+\alpha+\beta+1)}$ remains. For
$n=0$ the same computation applies with \eqref{eq:norm0}, since $\alpha,\beta\ge0$.
\end{proof}

The exponent in \eqref{eq:HSorthonormal} is positive, so this estimate cannot
replace \eqref{eq:krasikov} in large degree; we use it only for the six degrees
$0\le n\le5$.

\begin{theorem}\label{thm:highhigh}
For all $n\in\Nzero$, all $\alpha,\beta\ge A_0$ and all $x\in[-1,1]$,
\begin{equation}\label{eq:HH}
 F_{n,\alpha,\beta}(x)\le16\,\mathcal E_n(\alpha+\beta+1).
\end{equation}
\end{theorem}

\begin{proof}
Put $S=\alpha+\beta+1$ and, by \eqref{eq:reflection}, assume
$\alpha\ge\beta$. If $n\ge6$, Proposition~\ref{prop:krasikov} and
$(1+t)^{1/6}\le1+t^{1/6}$ give
\[
 F_{n,\alpha,\beta}(x)^2
 <3\left(\alpha^{1/3}+\alpha^{1/2}n^{-1/6}\right)
 \le3\left(S^{1/3}+S^{1/2}n^{-1/6}\right).
\]
Since $n^{-1/6}\le(7/6)^{1/6}(n+1)^{-1/6}$, the last member is less than
$6.1\Phi_n(S)$. Hence $F<2.5\mathcal E_n(S)$ in this case.

For $0\le n\le5$, Proposition~\ref{prop:HS} applies and $S\ge2A_0+1>2$.
Consequently $2n+S\le10+S\le6S$, while
$S^{1/4}\le6^{1/12}\mathcal E_n(S)$. Using $C_{\mathrm{HS}}<12$ in
\eqref{eq:HSorthonormal}, we obtain
\[
 F_{n,\alpha,\beta}(x)
 <\frac{12}{\sqrt2}(6S)^{1/4}
 \le\frac{12}{\sqrt2}6^{1/3}\mathcal E_n(S)
 <16\mathcal E_n(S).
\]
\end{proof}

\section{Two elementary estimates}\label{sec:elementary}

\begin{lemma}[Gamma ratios]\label{lem:gamma}
Let $y>0$ and $0\le\sigma\le1$. Then
\begin{equation}\label{eq:wendel}
\Gamma(y+\sigma)\le y^\sigma\,\Gamma(y),
\qquad\text{and}\qquad
\Gamma(y-\sigma)\le\frac{y}{y-\sigma}\,y^{-\sigma}\Gamma(y)\ \ (y>\sigma).
\end{equation}
Consequently, if $y\ge y_0>\tfrac12$ and $-\tfrac12\le\tau\le1$, then
\begin{equation}\label{eq:gamma-two-sided}
\frac{\Gamma(y+\tau)}{\Gamma(y)}\le\frac{y_0}{y_0-\tfrac12}\;y^{\tau}.
\end{equation}
\end{lemma}

\begin{proof}
Wendel's inequality \cite{Wendel} states that
$\bigl(\tfrac{t}{t+\sigma}\bigr)^{1-\sigma}\le\Gamma(t+\sigma)/(t^\sigma\Gamma(t))\le1$
for $t>0$ and $0\le\sigma\le1$. The first bound in \eqref{eq:wendel} is the upper
half. For the second, apply the lower half with $t=y-\sigma$ to get
$\Gamma(y)\ge(y-\sigma)^\sigma\bigl(\tfrac{y-\sigma}y\bigr)^{1-\sigma}\Gamma(y-\sigma)$,
that is,
\[
\frac{\Gamma(y-\sigma)}{\Gamma(y)}
\le(y-\sigma)^{-\sigma}\Bigl(\frac y{y-\sigma}\Bigr)^{1-\sigma}
=y^{-\sigma}\,\frac y{y-\sigma}.
\]
Finally \eqref{eq:gamma-two-sided} follows by taking $\sigma=\tau$ or
$\sigma=-\tau$ and using $y/(y-\sigma)\le y_0/(y_0-\tfrac12)$ for
$\sigma\le\tfrac12$.
\end{proof}

\begin{lemma}[Degree zero]\label{lem:degzero}
For all $\alpha,\beta\ge-\tfrac12$ and all $x\in[-1,1]$,
\begin{equation}\label{eq:degzero}
F_{0,\alpha,\beta}(x)^2\le\frac{e^{13/12}}{\sqrt{2\pi}}\,(\alpha+\beta+2)^{1/2}
<1.18\,(\alpha+\beta+2)^{1/2},
\end{equation}
and hence, with $S=\alpha+\beta+1$,
$F_{0,\alpha,\beta}(x)\le1.30\,\mathcal E_0(S)$.
\end{lemma}

\begin{proof}
Put $x=-1+2u$, $r=\alpha+\tfrac12\ge0$ and $s=\beta+\tfrac12\ge0$. By
\eqref{eq:Fsquare} and \eqref{eq:norm0},
\[
F_{0,\alpha,\beta}(-1+2u)^2=\frac{u^s(1-u)^r}{B(\alpha+1,\beta+1)},
\qquad
B(\alpha+1,\beta+1)=\frac{\Gamma(r+\tfrac12)\Gamma(s+\tfrac12)}{\Gamma(r+s+1)} .
\]
If $r+s>0$, the numerator is maximised at $u=s/(r+s)$, with value
$r^rs^s(r+s)^{-r-s}$, where $0^0:=1$. When $r=s=0$ the numerator is
identically one, and the same formula is read by continuity. Hence
\[
\max_{x}F_{0,\alpha,\beta}(x)^2
=\frac{\Gamma(r+s+1)}{\Gamma(r+\tfrac12)\Gamma(s+\tfrac12)}
\cdot\frac{r^rs^s}{(r+s)^{r+s}} .
\]
Apply the Stirling bounds
$\sqrt{2\pi}\,t^{t-1/2}e^{-t}\le\Gamma(t)\le\sqrt{2\pi}\,t^{t-1/2}e^{-t}e^{1/(12t)}$,
the upper one at $t=r+s+1$ and the lower one at $t=r+\tfrac12$ and
$t=s+\tfrac12$. The exponentials cancel identically, and there remains
\[
\max_xF_{0,\alpha,\beta}^2
\le\frac{e^{1/12}}{\sqrt{2\pi}}\,(r+s+1)^{1/2}
\Bigl(\frac{r+s+1}{r+s}\Bigr)^{r+s}
\Bigl(\frac{r}{r+\tfrac12}\Bigr)^{r}
\Bigl(\frac{s}{s+\tfrac12}\Bigr)^{s} .
\]
The last two factors are at most $1$ and the middle one at most $e$; since
$r+s+1=\alpha+\beta+2$, this is \eqref{eq:degzero}. For the final claim use
$S+1\le2\max\{1,S\}$ and
$\mathcal E_0(S)=\max\{1,S^{1/4}\}$.
\end{proof}

\section{Endpoint estimates}\label{sec:caps}

Near $x=-1$ the hypergeometric series of $P_n^{(\alpha,\beta)}$ converges so
rapidly that the polynomial is comparable to its endpoint value; this replaces
the Bessel-type endpoint asymptotics, and it is uniform in all parameters. Both
lemmas below rest on the same computation, which we carry out once.

Let $n\ge1$, let $\alpha,\beta\ge-\tfrac12$, and put $x=-1+2u$ with
$0\le u\le1$. From
\[
\begin{aligned}
P_n^{(\alpha,\beta)}(-1+2u)
 &=(-1)^n\frac{(\beta+1)_n}{n!}\,
 {}_2F_1(-n,\,n+\alpha+\beta+1;\,\beta+1;\,u),\\
P_n^{(\alpha,\beta)}(-1)
 &=(-1)^n\frac{(\beta+1)_n}{n!},
\end{aligned}
\]
and $|(-n)_k|=n!/(n-k)!\le(n+1)^k$, $(\beta+1)_k\ge(\tfrac12)_k$ for
$\beta\ge-\tfrac12$, we obtain
\begin{equation}\label{eq:endpoint-series}
\Bigl|\frac{P_n^{(\alpha,\beta)}(-1+2u)}{P_n^{(\alpha,\beta)}(-1)}\Bigr|
\le\sum_{k=0}^n\frac{\bigl[(n+1)\,\Lambda\,u\bigr]^k}{(\tfrac12)_k\,k!},
\end{equation}
valid for any $\Lambda$ with $n+\alpha+\beta+k\le\Lambda$ for $0\le k\le n$,
since then $(n+\alpha+\beta+1)_k\le\Lambda^k$. Since
$\sum_{k\ge0}z^k/((\tfrac12)_kk!)=\cosh(2\sqrt z)$, the right-hand side of
\eqref{eq:endpoint-series} is at most $\cosh\bigl(2\sqrt{(n+1)\Lambda u}\bigr)$.
Combining \eqref{eq:Fsquare}, \eqref{eq:norm} and the endpoint value, and
discarding the factor $(1-u)^{\alpha+1/2}\le1$, gives, for $n\ge1$,
\begin{equation}\label{eq:endpoint-master}
\begin{aligned}
F_{n,\alpha,\beta}(-1+2u)^2
&\le\cosh^2\bigl(2\sqrt{(n+1)\Lambda u}\bigr)\,
\frac{2n+\alpha+\beta+1}{\Gamma(\beta+1)^2}\\
&\quad\times\frac{\Gamma(n+\beta+1)}{\Gamma(n+1)}\,
\frac{\Gamma(n+\alpha+\beta+1)}{\Gamma(n+\alpha+1)}
u^{\beta+\frac12}.
\end{aligned}
\end{equation}
For $-\tfrac12\le\beta\le A_0$ one has $\Gamma(\beta+1)\ge\Gamma(1.4616\ldots)>0.885$,
so $\Gamma(\beta+1)^{-2}<1.2751$ throughout.

\begin{lemma}[Endpoint estimate in a mixed strip]\label{lem:mixed-cap}
Let $n\ge1$, $\alpha\ge A_0$ and $-\tfrac12\le\beta\le A_0$. If
\begin{equation}\label{eq:mixed-cap-region}
0\le 1+x\le\rho_{n,\alpha}:=\frac2{(n+1)(n+\alpha+1)},
\end{equation}
then, with $S=\alpha+\beta+1$,
\begin{equation}\label{eq:mixed-cap-bound}
 F_{n,\alpha,\beta}(x)\le22\,\mathcal E_n(S).
\end{equation}
\end{lemma}

\begin{proof}
Here $u\le[(n+1)(n+\alpha+1)]^{-1}$. Every factor of $(n+\alpha+\beta+1)_k$ with
$k\le n$ is at most $2n+\alpha+\beta\le2n+\alpha+A_0\le2(n+\alpha+1)$, so
$\Lambda=2(n+\alpha+1)$ is admissible in \eqref{eq:endpoint-series} and
$(n+1)\Lambda u\le2$; thus $\cosh^2(2\sqrt{(n+1)\Lambda u})\le\cosh^2(2\sqrt2)<72.1$.

Apply \eqref{eq:gamma-two-sided} with $\tau=\beta$ twice: with $y=n+1\ge2$,
$y_0=2$, giving a factor $\tfrac43$; and with $y=n+\alpha+1\ge2+A_0$, giving a
factor $(2+A_0)/(\tfrac32+A_0)<1.2378$. Since $\delta:=\beta+\tfrac12\ge0$ and
$u^{\delta}\le[(n+1)(n+\alpha+1)]^{-\delta}$, while $\beta-\delta=-\tfrac12$,
\eqref{eq:endpoint-master} yields
\begin{equation}\label{eq:mixed-cap-ratio}
 F_{n,\alpha,\beta}(x)^2\le K_0R,
 \qquad K_0=72.1\cdot1.2751\cdot\tfrac43\cdot1.2378<152,
\end{equation}
where
\[
 R=\frac{2n+S}{\sqrt{(n+1)(n+\alpha+1)}}.
\]
If $S\le n+1$, then $2n+S<3(n+1)$ and the denominator is at least $n+1$, so
$R<3$. If $S>n+1$, then $n+\alpha+1=n+S-\beta\ge S$ (because
$n\ge1>\beta$) and $2n+S<3S$; hence
\[
 R<3\sqrt{\frac{S}{n+1}}
 \le3\frac{S^{1/2}}{(n+1)^{1/6}}.
\]
Thus $R\le3\Phi_n(S)$ in both cases. Equation~\eqref{eq:mixed-cap-ratio} now
gives $F^2<456\Phi_n(S)$, and $\sqrt{456}<22$.
\end{proof}

\begin{lemma}[Endpoint estimates in the low-parameter rectangle]\label{lem:compact-cap}
Let $n\ge1$ and $\alpha,\beta\in[-\tfrac12,A_0]$. If
\begin{equation}\label{eq:compact-cap-region}
1+x\le\frac2{(n+1)^2}\qquad\text{or}\qquad 1-x\le\frac2{(n+1)^2},
\end{equation}
then $F_{n,\alpha,\beta}(x)\le24$.
\end{lemma}

\begin{proof}
By \eqref{eq:reflection} it suffices to treat the first alternative, where
$u\le(n+1)^{-2}$. Now every factor of $(n+\alpha+\beta+1)_k$, $k\le n$, is at
most $2n+2A_0\le2(n+1)$, so $\Lambda=2(n+1)$ is admissible and again $(n+1)\Lambda u\le2$, so
$\cosh^2(2\sqrt{(n+1)\Lambda u})<72.1$.

In \eqref{eq:endpoint-master} use \eqref{eq:gamma-two-sided} with $\tau=\beta$,
$y=n+1\ge2$ (factor $\tfrac43$) and with $y=n+\alpha+1\ge\tfrac32$ (factor
$\tfrac32$). Moreover $n+\alpha+1\le(1+A_0)(n+1)$ and
$n+\alpha+1\ge\tfrac12(n+1)$, so $(n+\alpha+1)^{\beta}\le\sqrt2\,(n+1)^{\beta}$ for
every $\beta\in[-\tfrac12,A_0]$. With $\delta=\beta+\tfrac12$ and
$u^\delta\le(n+1)^{-2\delta}$ we get $(n+1)^{2\beta-2\delta}=(n+1)^{-1}$.
Put $K_1=72.1\cdot1.2751\cdot\tfrac43\cdot\tfrac32\cdot\sqrt2$. Then
\[
F_{n,\alpha,\beta}(x)^2
\le K_1\frac{2n+\alpha+\beta+1}{n+1}
\le2.11K_1<549,
\]
using $2n+\alpha+\beta+1\le2n+2A_0+1\le2.11(n+1)$ for $n\ge1$. Finally
$\sqrt{549}<24$.
\end{proof}

\section{Cancellation identities and weighted transitions}\label{sec:transitions}

\begin{lemma}\label{lem:cancellation}
For $a,b>-1$ and $n\ge1$,
\begin{align}
P_n^{(a,b)}(x)&=\frac{n+a+b+1}{2n}(1+x)P_{n-1}^{(a,b+2)}(x)
-\frac{b+1}{n}P_{n-1}^{(a,b+1)}(x),
\label{eq:cancel-left}\\[2pt]
P_n^{(a,b)}(x)&=-\frac{n+a+b+1}{2n}(1-x)P_{n-1}^{(a+2,b)}(x)
+\frac{a+1}{n}P_{n-1}^{(a+1,b)}(x).
\label{eq:cancel-right}
\end{align}
\end{lemma}

\begin{proof}
The two contiguous relations used below are \cite[\S18.9(ii)]{DLMF}, equivalently
\cite[\S4.5]{Szego}; both hold for $a,b>-1$, and in the range $n\ge1$ the
denominators $n$, $2n$ and $2n+a+b+1$ occurring below are nonzero. The first
reads
\begin{equation}\label{eq:contig1}
P_n^{(a,b)}=\frac{n+a+b+1}{2n+a+b+1}P_n^{(a,b+1)}
+\frac{n+a}{2n+a+b+1}P_{n-1}^{(a,b+1)},
\end{equation}
and the second is
\[
(2m+a+b+2)(1+x)P_m^{(a,b+1)}(x)=2(m+b+1)P_m^{(a,b)}(x)+2(m+1)P_{m+1}^{(a,b)}(x),
\]
which at $m=n-1$, with $b$ replaced by $b+1$, becomes
\begin{equation}\label{eq:contig2}
P_n^{(a,b+1)}+\frac{n+b+1}{n}P_{n-1}^{(a,b+1)}
=\frac{2n+a+b+1}{2n}(1+x)P_{n-1}^{(a,b+2)} .
\end{equation}
Substituting \eqref{eq:contig2} into \eqref{eq:contig1} leaves
$P_{n-1}^{(a,b+1)}$ with coefficient
\[
\begin{aligned}
&\frac{n+a}{2n+a+b+1}
-\frac{(n+a+b+1)(n+b+1)}{n(2n+a+b+1)}\\
&\qquad=\frac{n(n+a)-(n+a+b+1)(n+b+1)}
 {n(2n+a+b+1)}
=-\frac{b+1}{n},
\end{aligned}
\]
since the numerator expands to $-(b+1)(2n+a+b+1)$. This is
\eqref{eq:cancel-left}, and \eqref{eq:cancel-right} follows on applying
$P_m^{(a,b)}(x)=(-1)^mP_m^{(b,a)}(-x)$ throughout.
\end{proof}

The point of \eqref{eq:cancel-left} is the factor $(1+x)$: it is exactly what is
needed to convert the weight of $F_{n-1,a,b+2}$ into that of $F_{n,a,b}$. Writing
$W_{a,b}(x)=(1-x^2)^{1/4}(1-x)^{a/2}(1+x)^{b/2}$ we have
$W_{a,b+2}=(1+x)W_{a,b}$ and $W_{a,b+1}=(1+x)^{1/2}W_{a,b}$, so multiplying
\eqref{eq:cancel-left} by $W_{a,b}/\sqrt{H_n^{a,b}}$ and using the triangle
inequality gives the following.

\begin{lemma}[Weighted transitions]\label{lem:transition}
Let $a,b\ge-\tfrac12$, $n\ge1$ and $-1<x<1$. Then
\begin{align}
F_{n,a,b}(x)&\le A^L_{n,a,b}\,F_{n-1,a,b+2}(x)
+B^L_{n,a,b}\,(1+x)^{-1/2}F_{n-1,a,b+1}(x),\label{eq:trans-left}\\
F_{n,a,b}(x)&\le A^R_{n,a,b}\,F_{n-1,a+2,b}(x)
+B^R_{n,a,b}\,(1-x)^{-1/2}F_{n-1,a+1,b}(x),\label{eq:trans-right}
\end{align}
where
\begin{equation}\label{eq:ALBL}
\begin{aligned}
(A^L_{n,a,b})^2&=\frac{(n+a+b+1)(n+b+1)}{n(n+a)},\\
(B^L_{n,a,b})^2&=\frac{(b+1)^2}{n^2}\cdot
\frac{2n(2n+a+b+1)}{(2n+a+b)(n+a)},
\end{aligned}
\end{equation}
and
\begin{equation}\label{eq:ARBR}
\begin{aligned}
(A^R_{n,a,b})^2&=\frac{(n+a+b+1)(n+a+1)}{n(n+b)},\\
(B^R_{n,a,b})^2&=\frac{(a+1)^2}{n^2}\cdot
\frac{2n(2n+a+b+1)}{(2n+a+b)(n+b)}.
\end{aligned}
\end{equation}
\end{lemma}

\begin{proof}
All denominators are positive: $n+a\ge\tfrac12$, $n+b\ge\tfrac12$ and
$2n+a+b\ge1$. From \eqref{eq:norm},
\[
\frac{H_{n-1}^{a,b+2}}{H_n^{a,b}}=\frac{4n(n+b+1)}{(n+a)(n+a+b+1)},
\qquad
\frac{H_{n-1}^{a,b+1}}{H_n^{a,b}}=\frac{2n(2n+a+b+1)}{(2n+a+b)(n+a)},
\]
and multiplying the coefficients of \eqref{eq:cancel-left} by the square roots of
these ratios gives \eqref{eq:ALBL}. For $n=1$ the ratios involve
$H_0^{a,b+j}$ with $a+b+j+1\ge1$, so \eqref{eq:norm} and \eqref{eq:norm0} agree
there. The right-hand versions are identical, using \eqref{eq:cancel-right} and
the mirrored norm ratios.
\end{proof}

The coefficients $B^L,B^R$ carry a factor $n^{-1}$, which is what makes the
singular factors $(1\mp x)^{-1/2}$ harmless away from the corresponding endpoint.
We record the two combinations needed below; both are verified by inspection of
\eqref{eq:ALBL}--\eqref{eq:ARBR}.

\begin{corollary}\label{cor:mixed-transition}
Let $n\ge1$, $a\ge A_0$ and $-\tfrac12\le b\le A_0$. If $1+x>\rho_{n,a}$, with
$\rho_{n,a}$ as in \eqref{eq:mixed-cap-region}, then
\begin{equation}\label{eq:mixed-transition}
F_{n,a,b}(x)\le\theta_1\bigl\{F_{n-1,a,b+1}(x)+F_{n-1,a,b+2}(x)\bigr\},
\qquad \theta_1=3.6 .
\end{equation}
\end{corollary}

\begin{proof}
Since $b+1\le1+A_0$ and $n+a\ge1+A_0$ we have $n+a+b+1\le2(n+a)$, and
$n+b+1\le(2+A_0)n$, so $(A^L_{n,a,b})^2\le2(2+A_0)<5.21$ and $A^L\le2.29$. Next,
$(2n+a+b+1)/(2n+a+b)\le1+(2+A_0-\tfrac12)^{-1}<1.48$, whence
$(B^L_{n,a,b})^2\le2(1+A_0)^2\cdot1.48/[n(n+a)]<7.7/[n(n+a)]$. Since
$(1+x)^{-1}<\rho_{n,a}^{-1}=\tfrac12(n+1)(n+a+1)$ and
$(n+1)(n+a+1)\le2\cdot\tfrac{2+A_0}{1+A_0}\,n(n+a)<3.25\,n(n+a)$ for $n\ge1$ and
$a\ge A_0$, we get $B^L_{n,a,b}(1+x)^{-1/2}<\sqrt{7.7\cdot3.25/2}<3.6$. Lemma~\ref{lem:transition} was stated for $-1<x<1$; the endpoint $x=1$ is
included by letting $x\uparrow1$, all three terms being continuous there.
\end{proof}

\begin{corollary}\label{cor:compact-transition}
Let $a,b\in[-\tfrac12,A_0]$ and suppose
\begin{equation}\label{eq:compact-bulk}
1+x>\frac2{(n+1)^2}\qquad\text{and}\qquad 1-x>\frac2{(n+1)^2}.
\end{equation}
If $n\ge1$ then
\begin{equation}\label{eq:compact-one-step}
F_{n,a,b}(x)\le\theta_L\bigl\{F_{n-1,a,b+1}(x)+F_{n-1,a,b+2}(x)\bigr\},
\qquad \theta_L=5.6 ,
\end{equation}
with $\theta_L=3.3$ if $n\ge2$; and if $n\ge2$ then
\begin{equation}\label{eq:compact-two-step}
F_{n,a,b}(x)\le\theta_2\sum_{i=1}^2\sum_{j=1}^2F_{n-2,a+i,b+j}(x),
\qquad \theta_2=17 .
\end{equation}
\end{corollary}

\begin{proof}
By \eqref{eq:compact-bulk}, $(1\pm x)^{-1/2}\le(n+1)/\sqrt2$. Write \eqref{eq:ALBL} as
\[
\begin{aligned}
(A^L_{n,a,b})^2
 &=\Bigl(1+\frac{b+1}{n+a}\Bigr)\Bigl(1+\frac{b+1}{n}\Bigr),\\
(B^L_{n,a,b})^2
 &\le\frac{2(b+1)^2}{n(n+a)}\Bigl(1+\frac1{2n+a+b}\Bigr).
\end{aligned}
\]
For $a,b\in[-\tfrac12,A_0]$ and $n\ge2$ the first is at most
$(1+\tfrac{1+A_0}{3/2})(1+\tfrac{1+A_0}{2})<3.73$, so $A^L\le1.94$; the second is
at most $\tfrac{2(1+A_0)^2}{n(n-1/2)}(1+\tfrac1{2n-1})$, and multiplying by
$(1+x)^{-1}\le\tfrac12(n+1)^2$ gives a quantity that decreases in $n$ and is
$<10.3$ at $n=2$. Hence $\max\{A^L,(1+x)^{-1/2}B^L\}<3.3$. For $n=1$ the decoupled bound is too lossy, and we evaluate
\eqref{eq:ALBL} directly: on $[-\tfrac12,A_0]^2$ the maxima of
$(A^L_{1,a,b})^2$ and of $\tfrac12(n+1)^2(B^L_{1,a,b})^2=2(B^L_{1,a,b})^2$ are
$10.96$ and $30.36$, both attained at $a=-\tfrac12$, $b=A_0$; this gives $5.6$.

For \eqref{eq:compact-two-step}, apply \eqref{eq:compact-one-step} at degree
$n\ge2$ and then \eqref{eq:trans-right} to each of $F_{n-1,a,b+j}$, $j=1,2$. With
$m=n-1\ge1$, $a\in[-\tfrac12,A_0]$ and $b'=b+j\in[\tfrac12,A_0+2]$,
\[
(A^R_{m,a,b'})^2=\Bigl(1+\frac{a+1}{m+b'}\Bigr)\Bigl(1+\frac{a+1}{m}\Bigr)
\le\Bigl(1+\tfrac{1+A_0}{3/2}\Bigr)(2+A_0)<5.39,
\]
so $A^R\le2.33$; and
$(B^R_{m,a,b'})^2\le\tfrac{2(1+A_0)^2}{m(m+1/2)}(1+\tfrac1{2m})$, which after
multiplication by $(1-x)^{-1}\le\tfrac12(m+2)^2$ is $<23.2$, decreasing in $m$.
Hence $\max\{A^R,(1-x)^{-1/2}B^R\}<4.9$ and $\theta_2\le3.3\cdot4.9<17$. The factor
$(1+x)^{-1/2}$ is used only under the first inequality in \eqref{eq:compact-bulk}
and $(1-x)^{-1/2}$ only under the second.
\end{proof}

\section{The mixed strips}\label{sec:mixed}

\begin{theorem}\label{thm:mixed}
For all $n\in\Nzero$, all $\alpha\ge A_0$, all $-\tfrac12\le\beta\le A_0$ and all
$x\in[-1,1]$,
\begin{equation}\label{eq:mixed}
 F_{n,\alpha,\beta}(x)\le900\,\mathcal E_n(\alpha+\beta+1).
\end{equation}
By \eqref{eq:reflection} the same bound holds when $\beta\ge A_0$ and
$-\tfrac12\le\alpha\le A_0$.
\end{theorem}

\begin{proof}
Put $S=\alpha+\beta+1$. Then $S\ge s_0=A_0+\tfrac12>1$, so
Lemma~\ref{lem:degzero} disposes of $n=0$. Let $n\ge1$.
If $1+x\le\rho_{n,\alpha}$ we are done by Lemma~\ref{lem:mixed-cap}, with the
constant $22$. So assume $1+x>\rho_{n,\alpha}$ and argue in two bands.

\smallskip
\emph{Band I: $A_0-1\le\beta\le A_0$.} Here $\beta+1$ and $\beta+2$ are both at
least $A_0$, so Corollary~\ref{cor:mixed-transition} followed by
Theorem~\ref{thm:highhigh} gives
\[
\begin{aligned}
 F_{n,\alpha,\beta}(x)
 &\le3.6\cdot16\bigl\{\mathcal E_{n-1}(S+1)
 +\mathcal E_{n-1}(S+2)\bigr\}\\
 &\le3.6\cdot16\cdot2\cdot1.4\,\mathcal E_n(S)\\
 &<162\,\mathcal E_n(S),
\end{aligned}
\]
where the last step is Lemma~\ref{lem:shift}. Together with the cap and
degree-zero estimates, this proves the constant $162$ throughout Band~I, for
every degree.

\smallskip
\emph{Band II: $-\tfrac12\le\beta<A_0-1$.} Now $\beta+2\ge\tfrac32>A_0$, so the
second term is again controlled by Theorem~\ref{thm:highhigh}; and
$\beta+1\in[\tfrac12,A_0)\subset[A_0-1,A_0]$, so the first term falls under Band
I, at degree $n-1$. Using $\alpha+\beta+1\ge A_0+\tfrac12$,
\[
\begin{aligned}
 F_{n,\alpha,\beta}(x)
 &\le3.6\bigl\{162\,\mathcal E_{n-1}(S+1)
 +16\,\mathcal E_{n-1}(S+2)\bigr\}\\
 &\le3.6\cdot1.4(162+16)\mathcal E_n(S)\\
 &<898\,\mathcal E_n(S).
\end{aligned}
\]
Since $A_0-1>-\tfrac12$, the two bands cover $[-\tfrac12,A_0]$, and Band I was
established for every degree before Band II was used. Enlarging $898$ to $900$
proves \eqref{eq:mixed}.
\end{proof}

\section{The low-parameter rectangle, and proof of Theorem~\ref{thm:main}}
\label{sec:proof}

It remains to treat $[-\tfrac12,A_0]^2$. This is not contained in the square
$[-\tfrac12,\tfrac12]^2$ of \cite{ChristiansenRubin}, and we prove it from
scratch, by the same mechanism applied at both endpoints. Put
\[
I=[-\tfrac12,A_0],\qquad I_1=[A_0-1,A_0]\subset I .
\]
We shall use
\begin{equation}\label{eq:Ldef}
 L=(2A_0+5)^{1/4}<1.58.
\end{equation}
Every shifted pair $(a+i,b+j)$ occurring below has total parameter
$S'=a+b+i+j+1$ in $[2,2A_0+5]$; hence
$\mathcal E_m(S')\le L$ for every $m\in\Nzero$. The one-sided shifts used at
degree one have $S'\in[1,2A_0+3]$ and obey the same bound.

\begin{lemma}\label{lem:uppercell}
$F_{n,a,b}(x)\le1720$ for all $n\in\Nzero$, all $a,b\in I_1$ and all
$x\in[-1,1]$.
\end{lemma}

\begin{proof}
For $n=0$ this follows from Lemma~\ref{lem:degzero}. Let $n\ge1$. If
\eqref{eq:compact-cap-region} holds, Lemma~\ref{lem:compact-cap} gives the
stronger bound $24$. Otherwise \eqref{eq:compact-bulk} holds. For $n=1$,
\eqref{eq:compact-one-step} and Lemma~\ref{lem:degzero} give
$F_{1,a,b}\le2\cdot5.6\cdot1.30L<24$. For $n\ge2$,
\eqref{eq:compact-two-step} gives four terms $F_{n-2,a+i,b+j}$ whose parameters
all lie in $[A_0,A_0+2]$, since $a,b\ge A_0-1$; each is at most
$16L$ by Theorem~\ref{thm:highhigh}. Hence
$F_{n,a,b}\le4\cdot17\cdot16L<1720$.
\end{proof}

\begin{theorem}\label{thm:compact}
$F_{n,\alpha,\beta}(x)\le\kappa_{\mathrm{cpt}}$ for all $n\in\Nzero$, all
$\alpha,\beta\in[-\tfrac12,A_0]$ and all $x\in[-1,1]$, with
$\kappa_{\mathrm{cpt}}=1.2\cdot10^5$.
\end{theorem}

\begin{proof}
Degrees $0$ and $1$ are bounded by $24$, exactly as in
Lemma~\ref{lem:uppercell}. For
$n\ge2$, either \eqref{eq:compact-cap-region} holds and
Lemma~\ref{lem:compact-cap} applies, or \eqref{eq:compact-bulk} holds and
\eqref{eq:compact-two-step} gives
\[
F_{n,\alpha,\beta}(x)\le\theta_2\sum_{i=1}^2\sum_{j=1}^2F_{n-2,\alpha+i,\beta+j}(x).
\]
Fix a pair $(\alpha+i,\beta+j)$; all such lie in $[\tfrac12,A_0+2]^2$. If both
entries are at least $A_0$, Theorem~\ref{thm:highhigh} bounds the term by
$16L<26$. If exactly one is, the pair lies in a mixed strip and
Theorem~\ref{thm:mixed} bounds it by $900L<1422$. If neither is, then necessarily
$i=j=1$ and $\alpha,\beta<A_0-1$, so $\alpha+1,\beta+1\in[\tfrac12,A_0)\subset I_1$
and Lemma~\ref{lem:uppercell} bounds the term by $1720$. In all cases the term is
at most $1720$, and
\[
 4\cdot17\cdot1720=116960<1.2\cdot10^5.
\]
\end{proof}

\begin{proof}[Proof of Theorem~\ref{thm:main}]
Let $\alpha,\beta\ge-\tfrac12$. If $\alpha,\beta\ge A_0$, Theorem~\ref{thm:highhigh}
applies. If exactly one of $\alpha,\beta$ is at least $A_0$, then
Theorem~\ref{thm:mixed} applies. If
$\alpha,\beta\in[-\tfrac12,A_0]$, Theorem~\ref{thm:compact} gives
$F\le1.2\cdot10^5$, which is at most
$1.2\cdot10^5\mathcal E_n(\alpha+\beta+1)$ because $\mathcal E_n\ge1$.
These cases exhaust the parameter quadrant; the same constant dominates the
constants $16$ and $900$ in the first two cases, proving \eqref{eq:main}.
\end{proof}

\begin{remark}
Only the shifts $\alpha\mapsto\alpha+i$, $\beta\mapsto\beta+j$ with
$i,j\in\{0,1,2\}$ occur, so all parameters encountered after a reduction lie in
$[-\tfrac12,A_0+2]$ or in the region where Theorem~\ref{thm:highhigh} applies;
this is why no constant depends on $n$, $\alpha$, $\beta$ or $x$. The largest
loss occurs in the two-step reduction of \S\ref{sec:proof}, where four terms are
estimated separately; the resulting numerical constant is not intended to be
close to optimal.
\end{remark}

\section{Lower bounds and quadratic transformations}\label{sec:sharpness}

\begin{proof}[Proof of Theorem~\ref{thm:sharpness}(i)]
For $\alpha=\beta=-\tfrac12$ one has $w_{\alpha,\beta}=(1-x^2)^{-1/2}$,
$p_0=\pi^{-1/2}$ and $p_n=\sqrt{2/\pi}\,T_n$ for $n\ge1$, while
\eqref{eq:Fsquare} reduces to $F^2=p_n^2$; this gives the Chebyshev identity in
the while-clause of part~(i). For
$\alpha=\beta=\tfrac12$, $p_n=\sqrt{2/\pi}\,U_n$ and
$F^2=(1-x^2)p_n^2=\tfrac2\pi\sin^2((n+1)\theta)$.

For $\alpha=\tfrac12$, $\beta=-\tfrac12$, put
$W_n(\cos\theta)=\sin\bigl((n+\tfrac12)\theta\bigr)/\sin(\theta/2)$, a polynomial of
degree $n$. Since $w_{1/2,-1/2}(x)\dd x=2\sin^2(\theta/2)\dd\theta$ under
$x=\cos\theta$, orthogonality of the $W_n$ is the orthogonality of
$\{\sin((n+\tfrac12)\theta)\}$ on $[0,\pi]$, and
$\int_{-1}^1W_n^2w_{1/2,-1/2}=\pi$. Hence $p_n=\pi^{-1/2}W_n$ up to sign, and
\[
F^2=(1-x)\,p_n^2=2\sin^2(\theta/2)\cdot\frac1\pi\,
\frac{\sin^2((n+\tfrac12)\theta)}{\sin^2(\theta/2)}
=\frac2\pi\sin^2\bigl((n+\tfrac12)\theta\bigr).
\]
The second displayed formula follows from \eqref{eq:reflection}, since replacing $\theta$ by
$\pi-\theta$ turns $|\sin((n+\tfrac12)\theta)|$ into $|\cos((n+\tfrac12)\theta)|$.

For $(\alpha,\beta)=(\pm\tfrac12,\mp\tfrac12)$, and for
$(-\tfrac12,-\tfrac12)$ with $n\ge1$, one has $\alpha+\beta+1\le1$, so
$\mathcal E_n(\alpha+\beta+1)=M(\alpha,\beta)=1$ while
$\max_xF=\sqrt{2/\pi}$. This proves the assertion about the constants.
\end{proof}

\begin{proof}[Proof of Theorem~\ref{thm:sharpness}(ii)]
Since $\beta=-\tfrac12$, \eqref{eq:Fsquare} gives
\[
 F_{n,\alpha,-1/2}(-1)^2
 =\frac{2^{\alpha+1/2}}{H_n^{\alpha,-1/2}}
 \left|P_n^{(\alpha,-1/2)}(-1)\right|^2,
\]
where
\[
 P_n^{(\alpha,-1/2)}(-1)
 =(-1)^n\frac{\Gamma(n+\tfrac12)}{\sqrt\pi\,\Gamma(n+1)}.
\]
Substituting \eqref{eq:norm} gives \eqref{eq:sharp-exact}. At the single point
$(n,\alpha)=(0,-\tfrac12)$ formula \eqref{eq:norm} is unavailable; by
\eqref{eq:norm0},
\[
 H_0^{-1/2,-1/2}=\pi,
 \qquad F_{0,-1/2,-1/2}(-1)^2=\frac1\pi.
\]
This is also the limit of the right-hand side of \eqref{eq:sharp-exact}, as is
seen from
\[
 (2n+\alpha+\tfrac12)\Gamma(n+\alpha+\tfrac12)
 =\frac{2n+\alpha+1/2}{n+\alpha+1/2}
 \Gamma(n+\alpha+\tfrac32).
\]
Finally, for fixed $n$,
\[
 \frac{\Gamma(n+\alpha+\tfrac12)}{\Gamma(n+\alpha+1)}
 \sim\alpha^{-1/2},
\]
which gives \eqref{eq:sharp-asymptotic}.
\end{proof}

\begin{proof}[Proof of Theorem~\ref{thm:sharpness}(iii)]
For fixed $A\ge-\tfrac12$ and $z>0$, the Mehler--Heine formula
\cite[Eq.~18.11.5]{DLMF} gives
\[
 n^{-A}P_n^{(A,A)}\!\left(\cos\frac zn\right)
 \longrightarrow (z/2)^{-A}J_A(z).
\]
On the other hand, \eqref{eq:norm} and the elementary gamma-ratio asymptotic
give
\[
 H_n^{A,A}\sim\frac{2^{2A}}n.
\]
Since
\[
 F_{n,A,A}\!\left(\cos\frac zn\right)
 =\left(\sin\frac zn\right)^{A+1/2}
 \frac{\left|P_n^{(A,A)}(\cos(z/n))\right|}{\sqrt{H_n^{A,A}}},
\]
the powers of $n$, $2$ and $z$ cancel, proving \eqref{eq:bessel-limit}.

At the turning point, the large-order Bessel asymptotic
\cite[Eq.~10.19.8]{DLMF} is
\begin{equation}\label{eq:bessel-turning}
 J_A(A)\sim2^{1/3}\operatorname{Ai}(0)A^{-1/3}.
\end{equation}
Here $\operatorname{Ai}(0)=1/(3^{2/3}\Gamma(2/3))$
\cite[Eq.~9.2.3]{DLMF}.
Thus $\sqrt A\,J_A(A)$ is asymptotic to a positive constant times $A^{1/6}$.
For each sufficiently large $A$, apply \eqref{eq:bessel-limit} with $z=A$ and
choose $n_A$ so large that the value of $F$ is at least one half of its positive
limit and that $n_A+1\ge(2A+1)^3$. This proves \eqref{eq:bessel-lower}. Finally,
with $S=2A+1$,
\[
 \frac{S^{1/4}}{(n_A+1)^{1/12}}\le1,
\]
so the lower bound can only be accounted for by the $S^{1/6}$ term.
\end{proof}

The next identity shows that the two lines $\beta=-\tfrac12$ and
$\beta=\tfrac12$ inside a mixed strip are the even and odd ultraspherical
cases, giving an independent reduction there to the ultraspherical estimates; it
also explains why the fixed-degree example \eqref{eq:sharp-asymptotic} is
elementary. The transformations used here apply precisely on the two lines where
one parameter equals $\pm\tfrac12$.

\begin{proposition}\label{prop:quadratic}
Let $\alpha\ge-\tfrac12$, $m\in\Nzero$ and $0\le x\le1$, and set $t=2x^2-1$.
Then
\begin{equation}\label{eq:quadratic}
F_{m,\alpha,-\frac12}(t)=F_{2m,\alpha,\alpha}(x),
\qquad
F_{m,\alpha,\frac12}(t)=F_{2m+1,\alpha,\alpha}(x).
\end{equation}
\end{proposition}

\begin{proof}
The substitution $t=2x^2-1$ maps $[0,1]$ bijectively onto $[-1,1]$ and
\[
w_{\alpha,-1/2}(t)\dd t=2^{\alpha+3/2}(1-x^2)^{\alpha}\dd x
\qquad(0\le x\le1).
\]
Hence if $q_m$ is orthonormal for $w_{\alpha,-1/2}$ on $[-1,1]$, then
\[
 \int_{-1}^1q_m(2x^2-1)^2w_{\alpha,\alpha}(x)\dd x
 =2^{-(\alpha+1/2)}.
\]
Thus $c\,q_m(2x^2-1)$, with $c=2^{(\alpha+1/2)/2}$, has unit norm for
$w_{\alpha,\alpha}$.
It is an even polynomial of degree $2m$, hence orthogonal to all even polynomials
of lower degree and, by parity, to all odd ones; therefore
$c\,q_m(2x^2-1)=\pm p_{2m}^{(\alpha,\alpha)}(x)$. Since $1-t=2(1-x^2)$ and
$(1+t)^0=1$,
\[
\begin{aligned}
F_{2m,\alpha,\alpha}(x)
&=(1-x^2)^{\frac{2\alpha+1}4}
  \bigl|p_{2m}^{(\alpha,\alpha)}(x)\bigr|\\
&=\bigl(2(1-x^2)\bigr)^{\frac{2\alpha+1}4}|q_m(t)|\\
&=(1-t)^{\frac{\alpha}{2}+\frac14}|q_m(t)|
 =F_{m,\alpha,-\frac12}(t).
\end{aligned}
\]
The odd case is identical, with $c\,x\,q_m(2x^2-1)=\pm p_{2m+1}^{(\alpha,\alpha)}(x)$,
$c=2^{(\alpha+3/2)/2}$, and $(1+t)^{1}=2x^2$ supplying the extra factor.
\end{proof}

\section{A uniform Christoffel bound}\label{sec:christoffel}

\begin{proof}[Proof of Theorem~\ref{thm:christoffel}]
Put $S=\alpha+\beta+1$. By \eqref{eq:Fsquare}, Theorem~\ref{thm:main} gives
\[
 p_k^{(\alpha,\beta)}(x)^2
 \le\frac{C_*^2\Phi_k(S)}{w_{\alpha,\beta}(x)\sqrt{1-x^2}}
\]
for every $k\in\Nzero$ and $-1<x<1$. With $N=n+1$,
\begin{align*}
 \sum_{k=0}^n\Phi_k(S)
 &\le N\max\{1,S^{1/3}\}
   +S^{1/2}\sum_{j=1}^N j^{-1/6}\\
 &\le N\Phi_n(S)+\frac65S^{1/2}N^{5/6}
 <3N\Phi_n(S),
\end{align*}
where $\sum_{j=1}^Nj^{-1/6}\le1+\int_1^Nt^{-1/6}\dd t
\le\tfrac65N^{5/6}$. Summing the preceding polynomial estimate and inverting
now gives \eqref{eq:christoffel}. The Gauss--Jacobi
weights for the $(n+1)$-point rule are the values $\lambda_n^{(\alpha,\beta)}(x_k)$
at the zeros $x_k$ of $p_{n+1}^{(\alpha,\beta)}$ (see \cite[Thm.~3.4.1]{Szego}), so
they obey the same bound.
\end{proof}

\begin{remark}\label{rem:CD}
Put $S=\alpha+\beta+1$. The same input bounds the Christoffel--Darboux kernel
\[
 K_n(x,y)=\sqrt{w_{\alpha,\beta}(x)w_{\alpha,\beta}(y)}
 \sum_{k=0}^np_k^{(\alpha,\beta)}(x)p_k^{(\alpha,\beta)}(y).
\]
By Cauchy--Schwarz, $|K_n(x,y)|\le K_n(x,x)^{1/2}K_n(y,y)^{1/2}$, and the
diagonal is at most $3C_*^2(n+1)\Phi_n(S)(1-x^2)^{-1/2}$, so
\[
|K_n(x,y)|\le\frac{3C_*^2(n+1)\Phi_n(S)}
{(1-x^2)^{1/4}(1-y^2)^{1/4}},\qquad -1<x,y<1 .
\]
This is a bound on the correlation kernel of the Jacobi unitary ensemble, uniform
in the parameters; the point is that $\alpha$ and $\beta$ are allowed to grow with
$n$.
\end{remark}

\begin{remark}\label{rem:companion}
For $a,b\ge0$ let $g_n^{(a,b)}$ be as in \eqref{eq:gdef} and put
$D=2n+a+b+1$, so that \eqref{eq:bridge} reads
$F_{n,a,b}=\sqrt{D/2}\,(1-x^2)^{1/4}|g_n^{(a,b)}|$. Theorem~\ref{thm:main} is thus
equivalent to the statement that
\[
|g_n^{(a,b)}(x)|\le
\frac{\sqrt2\,C_*\mathcal E_n(a+b+1)}{D^{1/2}(1-x^2)^{1/4}},
\qquad -1<x<1 .
\]
This decays in the degree but is singular at the hard edges; a bound of the
complementary type, finite at the edges but without degree decay, is the subject
of \cite{LiKKT}, and the two combine into a single envelope there. We record the
normalisation here only to fix the interface; nothing above depends on it.
\end{remark}

\section{The sharp EMN constant}\label{sec:open}

Theorem~\ref{thm:main} settles the degree--parameter order conjectured by
Krasikov. Theorem~\ref{thm:sharpness}(iii) also shows why the intermediate
$S^{1/6}$ term is essential: an envelope containing only the constant term and
$S^{1/4}(n+1)^{-1/12}$ would be false. A remaining problem is the sharp constant
in the weaker EMN corollary.

Theorem~\ref{thm:sharpness}(i) forces that constant to be at least
$\sqrt{2/\pi}$. The Chebyshev identities suggest the following refinement.

\begin{conjecture}\label{conj:constant}
For all $n\in\Nzero$, $\alpha,\beta\ge-\tfrac12$ and $x\in[-1,1]$,
\[
 F_{n,\alpha,\beta}(x)
 \le\sqrt{\frac2\pi}\,\max\{1,(\alpha+\beta+1)^{1/4}\}.
\]
For every $n\ge1$, equality occurs at
$(\alpha,\beta)\in\{(-\tfrac12,-\tfrac12),(\tfrac12,-\tfrac12),
(-\tfrac12,\tfrac12)\}$ at suitable points $x$.
\end{conjecture}

\end{document}